\documentclass[12pt]{amsart}

\usepackage{amssymb}
\usepackage{amscd}
\usepackage{statex}

\title{A new characterisation of the Fermat curve}
\author{Satoru Fukasawa}

\thanks{The author was partially supported by JSPS KAKENHI Grant Numbers JP19K03438 and JP22K03223}
\address{Faculty of Science, Yamagata University, Kojirakawa-machi 1-4-12, Yamagata 990-8560, Japan}
\email{s.fukasawa@sci.kj.yamagata-u.ac.jp}

\newtheorem*{theorem}{Theorem}
\newtheorem{proposition}{Proposition}

\newtheorem{corollary}{Corollary}
\newtheorem{fact}{Fact}

\theoremstyle{definition}

\begin{document}
\begin{abstract}
This paper presents a new characterisation of the Fermat curve, according to the arrangement of Galois points. 
\end{abstract}
\maketitle

\section{Introduction} 
Let $k$ be an algebraically closed field of characteristic zero and let $C \subset \mathbb{P}^2$ be an irreducible plane curve of degree $d \ge 3$ over $k$. 
The genus of the smooth model of $C$ is denoted by $g$. 
This paper proves the following. 

\begin{theorem}
Assume that $g \ge 1$. 
Then there exist non-collinear outer Galois points $P_1, P_2, P_3 \in \mathbb{P}^2 \setminus C$ for $C$ if and only if $C$ is projectively equivalent to the Fermat curve 
$$ X^d+Y^d+Z^d=0, $$
where $(X:Y:Z)$ is a system of homogeneous coordinates of $\mathbb{P}^2$. 
\end{theorem}

We recall the definition of a {\it Galois point}, which was introduced by Hisao Yoshihara in 1996 (\cite{fukasawa1, miura-yoshihara, yoshihara}). 
We consider the projection 
$$\pi_P: C \rightarrow \mathbb{P}^1; \ Q \mapsto (F_0(Q):F_1(Q))$$
from a point $P \in \mathbb{P}^2 \setminus C$, where $F_0, F_1$ are homogeneous polynomials of degree one defining $P$. 
The point $P \in \mathbb{P}^2 \setminus C$ is called an outer Galois point for $C$ if the function field extension $k(C)/\pi_P^*k(\mathbb{P}^1)$ induced by $\pi_P$ is Galois. 
Obviously, non-collinear points $(1:0:0)$, $(0:1:0)$ and $(0:0:1)$ are outer Galois points for the Fermat curve, namely, the ``if part'' of Theorem is confirmed.

\section{Preliminaries}

For different points $P, Q \in \mathbb{P}^2$, the line passing through $P$ and $Q$ is denoted by $\overline{PQ}$. 
Let $X$ be the smooth model of $C$ and let $g$ be the genus of $X$. 
In this paper, we assume that $g \ge 1$. 
The following well known fact on Galois extensions is used frequently (see, for example, \cite[III.7.1, III.8.2]{stichtenoth}). 

\begin{fact}
Let $\pi: X \rightarrow X'$ be a surjective morphism between smooth projective curves such that $k(X)/\pi^*k(X')$ is a Galois extension, and let $G \subset {\rm Aut}(X)$ be the induced Galois group. 
Then the following hold. 
\begin{itemize} 
\item[(a)] The Galois group $G$ acts on any fibre transitively. 
\item[(b)] If $Q \in X$, then $e_Q=|G(Q)|$, where $e_Q$ is a ramification index at $Q$, and $G(Q)$ is the stabiliser subgroup of $Q$ under $G$. 
\end{itemize}
\end{fact}

A well known theorem of Hurwitz (and its proof) on the order of the automorphism group plays an important role in this paper (see, for example, \cite[Theorem 11.56]{hkt}). 
The following version is needed later. 

\begin{fact}[Hurwitz's theorem] \label{Hurwitz} 
Let $G \subset {\rm Aut}(X)$. 
Assume that $g \ge 2$ and $X/G \cong \mathbb{P}^1$. 
The number of short orbits under $G$ is denoted by $n$, and the orders of their stabiliser subgroups are denoted by $r_1, \ldots, r_n$.  
Let 
$$T:=-2+\sum_{i=1}^n\left (1-\frac{1}{r_i}\right). $$
Then the following hold. 
\begin{itemize} 
\item[(a)] $T>0$ and $|G|=(2g-2)/T$. 
\item[(b)] $n \ge 3$. 
\item[(c)] If $n \ge 5$, then $|G| \le 4(g-1)$. 
\end{itemize} 
\end{fact}

For an outer Galois point $P \in \mathbb{P}^2 \setminus C$, the induced Galois group is denoted by $G_P$, which admits an injective homomorphism $G_P \hookrightarrow {\rm Aut}(X)$. 
For the case where two Galois points exist, the following holds. 

\begin{fact}[Lemma 7 in \cite{fukasawa2}] \label{trivial intersection}
Let $P_1, P_2$ be different outer Galois points. 
Then 
$$ G_{P_1} \cap G_{P_2}=\{1\}. $$
\end{fact}

Several important results in the theory of Galois points are needed for the proof. 

\begin{fact}[Yoshihara \cite{yoshihara}] \label{smooth}
If $C$ is smooth and there exist three outer Galois points, then $C$ is projectively equivalent to the Fermat curve. 
\end{fact}

\begin{fact}[Duyaguit--Miura \cite{duyaguit-miura}] \label{prime degree} 
If $g \ge 1$, $d$ is a prime and there exist three outer Galois points, then $C$ is projectively equivalent to the Fermat curve. 
\end{fact}

\begin{fact}[Fukasawa \cite{fukasawa3}] \label{unramified} 
If $g \ge 1$ and there exist three outer Galois points for $C$, then the morphism $X \rightarrow \mathbb{P}^2$ induced by the normalisation $X \rightarrow C$ is unramified. 
\end{fact}

\begin{fact}[Fukasawa \cite{fukasawa5}] \label{separated variables} 
Assume that points $P_1=(1:0:0)$, $P_2=(0:1:0)$ are outer Galois points, and that $|\langle G_{P_1}, G_{P_2} \rangle|<\infty$. 
Then there exist two polynomials $f_1(x), f_2(x) \in k[x]$ of degree $|\langle G_{P_1}, G_{P_2} \rangle|/d$ such that $C$ is an irreducible component of the curve $\tilde{C}$ defined by 
$$ f_1(x)-f_2(y)=0. $$ 
\end{fact}

Since points of $\tilde{C} \cap \{Z=0\}$ are given by $X^{\deg \tilde{C}}+c Y^{\deg \tilde{C}}=0$ for some $c \in k \setminus \{0\}$, it follows that $\tilde{C} \cap \{Z=0\}$ consists of $\deg \tilde{C}$ points. 
This implies the following. 

\begin{corollary} \label{transversal} 
Let $P_1, P_2$ be different outer Galois points. 
Assume that $|\langle G_{P_1}, G_{P_2} \rangle| < \infty$.  
Then the set $C \cap \overline{P_1P_2}$ consists of $d$ points, namely, all points of $C \cap \overline{P_1P_2}$ are smooth points of $C$.  
\end{corollary}

\begin{proposition} \label{orders} 
Assume that points $P_1=(1:0:0)$, $P_2=(0:1:0)$, $P_3=(0:0:1)$ are outer Galois points. 
If $|\langle G_{P_1}, G_{P_2} \rangle|=d^2$, $|\langle G_{P_1}, G_{P_3} \rangle|=d^2$, and $|\langle G_{P_2}, G_{P_3} \rangle| < \infty$, then $C$ is projectively equivalent to the Fermat curve.  
\end{proposition}

\begin{proof}
Since $|\langle G_{P_1}, G_{P_2} \rangle|=d^2$, it follows from Fact \ref{separated variables} that there exist polynomials $f_{0, 1}(x), f_{0, 2}(x) \in k[x]$ of degree $d$ such that $C$ is defined by $f_{0,1}(x)-f_{0, 2}(y)=0$. 
Let $c_1=f_{0, 1}(0)$ and let $c_2=f_{0, 2}(0)$. 
Since $P_3 \in \mathbb{P}^2 \setminus C$, it follows that $c_1-c_2 \ne 0$. 
We take $f_1(x):=(f_{0, 1}(x)-c_1)/(c_1-c_2)$ and $f_2(x):=(-f_{0, 2}(x)+c_2)/(c_1-c_2)$. 
Then $C$ is defined by
$$ f(x, y):=f_1(x)+f_2(y)+1=0.  $$
Since $|\langle G_{P_1}, G_{P_3} \rangle|=d^2$, there exist polynomials $g_1(u), g_2(u) \in k[u]$ of degree $d$ with $g_1(0)=g_2(0)=0$ such that $C$ is defined by 
$$ g(u, v):=g_1(u)+g_2(v)+1=0, $$
where $u=X/Y=x/y$ and $v=Z/Y=1/y$. 
Then 
$$ y=\frac{1}{v}, \ x=\frac{u}{v}$$
and define 
$$ h(u, v):=v^d f(u/v, 1/v)=v^d f_1(u/v)+v^d f_2(1/v)+v^d. $$
The degree $d$ homogeneous parts of $g(u, v)$ and of $h(u, v)$ are  $c_1u^d+c_2 v^d$ for some $c_1, c_2 \in k \setminus \{0\}$ and $v^d f_1(u/v)+v^d$ respectively. 
Since $g(u, v)$ and $h(u, v)$ are the same up to a constant, it follows that polynomials 
$$c_1u^d+c_2 v^d \ \mbox{ and } \ v^d f_1(u/v)+v^d$$ 
are the same up to a constant. 
Therefore, polynomials $c_1x^d+c_2$ and $f_1(x)+1$ are the same up to a constant. 
By taking into account the definition of $f(x, y)$, $C$ is defined by a polynomial of the form 
$$ x^d+f_3(y), $$
where $f_3(y) \in k[y]$ is of degree $d$. 
This implies that all singular points of $C$ are contained in the union of two lines $\{X=0\}$ and $\{Z=0\}$. 
According to Corollary \ref{transversal}, $C$ is smooth. 
It follows from Fact \ref{smooth} that $C$ is projectively equivalent to the Fermat curve. 
\end{proof} 

Assume that points $P_1, P_2, P_3 \in \mathbb{P}^2 \setminus C$ are non-collinear outer Galois points. 
Fix a triple $i, j, k$ such that $\{i, j, k\}=\{1, 2, 3\}$. 
We consider the group 
$$ \langle G_{P_i}, G_{P_j} \rangle \cap G_{P_k}. $$
The order is denoted by $m_k$.   
The following proposition is proved by the same method as in \cite[Proof of Theorem 1.5]{fukasawa4}; however, the proof is given for the convenience of the readers. 

\begin{proposition} \label{extendable}
Let $P_1=(1:0:0)$, $P_2=(0:1:0)$, $P_3=(0:0:1)$. 
Assume that $P_1, P_2, P_3$ are Galois points, and that $|\langle G_{P_1}, G_{P_2} \rangle|< \infty$. 
Then the group $\langle G_{P_1}, G_{P_2} \rangle \cap G_{P_3}$ is a cyclic group, there exists an injective homomorphism 
$$ \langle G_{P_1}, G_{P_2} \rangle \cap G_{P_3} \hookrightarrow PGL(3, k), $$
and a generator $\sigma \in \langle G_{P_1}, G_{P_2} \rangle \cap G_{P_3}$ is represented by a matrix 
$$ A_{\sigma}=\left(\begin{array}{ccc} 
1 & 0 & 0 \\
0 & 1 & 0 \\ 
0 & 0 & \zeta 
\end{array}\right), $$
where $\zeta \in k$ is a primitive $m_3$-th root of unity.  
In particular, for a point $Q \in \mathbb{P}^2$ and an integer $i$ with $1 \le i \le m_3-1$, $\sigma^i(Q)=Q$ if and only if $Q \in \overline{P_1P_2} \cup \{P_3\}$. 
\end{proposition} 

\begin{proof} 
By Corollary \ref{transversal}, the set $C \cap \overline{P_1P_2}$ consists of different $d$ points $Q_1, \ldots, Q_d$. 
Let 
$$ D_3:=Q_1+\cdots+Q_d. $$ 
Let $\sigma \in \langle G_{P_1}, G_{P_2} \rangle \cap G_{P_3}$.  
Since $\sigma \in \langle G_{P_1}, G_{P_2} \rangle$, it follows that $\sigma(Q_i) \in \overline{P_1P_2}$ for $i=1, \ldots, d$, namely, 
$$ \sigma^*D_3=D_3. $$
Let $D_1$ and $D_2$ be divisors coming from $C \cap \overline{P_3P_2}$ and $C \cap \overline{P_3P_1}$ respectively. 
We take a function $f \in k(X)$ with $k(f)=k(X)^{G_{P_3}}$ such that 
$$ (f)=D_1-D_2. $$ 
Similarly, we can take a function $g \in k(X)^{G_{P_1}}$ such that 
$$ (g)=D_3-D_2. $$
Since $\overline{P_1P_2}$ does not pass through $P_3$, it follows that $g \not\in \langle 1, f \rangle \subset \mathcal{L}(D_2)$. 
It follows from the condition $\sigma^*D_3=D_3$ that $\sigma^*g=a(\sigma) g$ for some $a(\sigma) \in k$. 
Therefore, a linear space $\langle 1, f, g \rangle \subset \mathcal{L}(D_2)$ is invariant under the action of $\sigma$. 
Since the embedding $X \rightarrow C \subset \mathbb{P}^2$ is represented by $(1:f:g)$, there exists an injective homomorphism 
$$ \langle G_{P_1}, G_{P_2} \rangle \cap G_{P_3} \hookrightarrow PGL(3, k); \ \sigma \mapsto 
\left(\begin{array}{ccc}
1 & 0 & 0 \\
0 & 1 & 0 \\
0 & 0 & a(\sigma) 
\end{array}\right). 
$$ 
The map $\langle G_{P_1}, G_{P_2} \rangle \cap G_{P_3} \rightarrow k \setminus \{0\}; \sigma \mapsto a(\sigma)$ is an injective homomorphism. 
This implies that $\langle G_{P_1}, G_{P_2} \rangle \cap G_{P_3}$ is a cyclic group, and that $C$ is invariant under the linear transformation $(X:Y:Z) \mapsto (X:Y:\zeta Z)$. 
\end{proof}

\section{Proof of Theorem: The case $d \ne 4$} 
Hereafter, we assume that points $P_1, P_2, P_3 \in \mathbb{P}^2 \setminus C$ are non-collinear outer Galois points. 
According to Fact \ref{smooth}, we can assume that $C$ is a singular curve of degree $d \ge 3$, namely, $g < (d-1)(d-2)/2$.  
By Fact \ref{prime degree}, we can assume that $d$ is not a prime. 
In the proof below, we would like to prove that $d < 6$. 
The proof for the case where $d=4$ is carried out in the next section. 

For any $i, j$ with $i \ne j$, the group $\langle G_{P_i}, G_{P_j} \rangle$ acts on $C \cap \overline{P_iP_j}$. 
It follows form Fact \ref{trivial intersection} that $|\langle G_{P_i}, G_{P_j} \rangle| \ge d^2$. 
The orbit-stabiliser theorem implies that $|\langle G_{P_i}, G_{P_j} \rangle|<\infty$ (even if $g=1$), and that the order of the stabiliser subgroup $\langle G_{P_i}, G_{P_j} \rangle(Q)$ of a point $Q \in C \cap \overline{P_1P_2}$ is at least $d$. 
By Proposition \ref{orders}, if $|\langle G_{P_i}, G_{P_j}\rangle(Q)|= d$ for any pair $i, j \in \{1, 2, 3\}$ with $i \ne j$, then $C$ is projectively equivalent to the Fermat curve.  
Therefore, there exist $i, j$ such that $|\langle G_{P_i}, G_{P_j}\rangle(Q)| \ge d+1$. 
Assume that $g=1$. 
A well known theorem on the automorphism group of an elliptic curve implies that $d+1 \le 6$ (see, for example, \cite[III. Theorem 10.1]{silverman}). 
Therefore, we can assume that $g \ge 2$. 

Let   
$$ G:=\langle G_{P_1}, G_{P_2}, G_{P_3} \rangle. $$
There are two main steps in the proof; $m_k=1$ for some $k$, and $m_k \ge 2$ for any $k$. 
For each step, we take into consideration the various possibilities for the number of short orbits and the orders of their stabiliser subgroups. 

{\it Case 1: $m_k=1$ for some $k$.} 
In this case, $|G| \ge d^3$. 
We can assume that $k=3$, namely, $\langle G_{P_1}, G_{P_2} \rangle \cap G_{P_3}=\{1\}$. 
As we saw above, for a point $Q \in C \cap \overline{P_1P_2}$, $|G(Q)| \ge d$. 

Assume that the number of short orbits under $G$ is at least five. 
It follows from Fact \ref{Hurwitz} (c) that 
$$ d^3 \le |G| \le 4(g-1) \le 2d(d-3). $$
Then 
$$ d \le 2 \times \frac{d-3}{d} < 2 \times 1=2. $$
Assume that the number of short orbits is four. 
On the estimate of $T$ in Fact \ref{Hurwitz}, the assumption that three stabilisers have order $2$ and the fourth has order $d$, from $|G(Q)| \ge d$, gives 
$$ T \ge -2+\left(1-\frac{1}{2}\right) \times 3+\left(1-\frac{1}{d}\right)=\frac{1}{2}-\frac{1}{d}. $$
It follows from Fact \ref{Hurwitz} that 
$$ d^3 \le |G| = \frac{2g-2}{T} \le \frac{2d^2(d-3)}{d-2}. $$
Then 
$$ d \le 2 \times \frac{d-3}{d-2} < 2. $$

Assume that the number of short orbits is three. 
We take a smooth point $Q_k \in C \cap \overline{P_iP_j}$ for any $\{i, j, k\}=\{1, 2, 3\}$. 
We consider the case where three orbits $G \cdot Q_1$, $G \cdot Q_2$, $G \cdot Q_3$ are different. 
By Proposition \ref{orders}, two of $|G(Q_1)|$, $|G(Q_2)|$, $|G(Q_3)|$ are at least $d+1$. 
It follows from Fact \ref{Hurwitz} that 
$$ T \ge -2+\left(1-\frac{1}{d}\right)+\left(1-\frac{1}{d+1}\right) \times 2=\frac{d^2-2d-1}{d(d+1)}, $$
and 
$$ d^3 \le |G| = \frac{2g-2}{T} \le \frac{d^2(d+1)(d-3)}{d^2-2d-1}. $$
Then 
$$ d \le \frac{d^2-2d-3}{d^2-2d-1} <1.$$
We consider the case where there exist $i, j, k$ with $\{i, j, k\}=\{1, 2, 3\}$ such that $G \cdot Q_i=G \cdot Q_j$ and $G \cdot Q_i \ne G \cdot Q_k$. 
By Proposition \ref{orders}, $|G(Q_i)|=|G(Q_j)| \ge d+1$ and $|G(Q_k)| \ge d$.  
Then 
$$ T \ge -2+\left(1-\frac{1}{2}\right)+\left(1-\frac{1}{d}\right)+\left(1-\frac{1}{d+1}\right)=\frac{d^2-3d-2}{2d(d+1)}, $$
and 
$$ d^3 \le |G| = \frac{2g-2}{T} \le \frac{2d^2(d+1)(d-3)}{d^2-3d-2}. $$
It follows that if $d \ge 5$, then 
$$ d \le 2 \times \frac{d^2-2d-3}{d^2-3d-2} <4.$$

Finally, we consider the case where $G \cdot Q_1=G \cdot Q_2=G \cdot Q_3$. 
By Proposition \ref{orders}, $|G(Q_3)| \ge d+1$. 
Assume that the projection $\pi_{P_3}$ from $P_3$ is ramified at some point $Q \in C \cap \overline{P_1P_2}$. 
Since $\langle G_{P_1}, G_{P_2} \rangle \cap G_{P_3}=\{1\}$, it follows that $|G(Q)| \ge 2d$. 
It follows that 
$$ T \ge -2+\left(1-\frac{1}{2}\right)+\left(1-\frac{1}{3}\right)+\left(1-\frac{1}{2d}\right)=\frac{d-3}{6d}, $$
and 
$$ d^3 \le |G| = \frac{2g-2}{T} \le d(d-3) \times \frac{6d}{d-3}=6d^2. $$
This implies that $d \le 6$. 
If $d=6$, then $g=10$, namely, $C$ is smooth. 
This is a contradiction. 
It follows that $d < 6$. 
Therefore, we can assume that $\pi_{P_3}$ is not ramified at any point in $C \cap \overline{P_1P_2}$.
Then $G_{P_3} \cdot Q$ contains $d$ points for any point $Q \in C \cap \overline{P_1P_2}$.  
Note that $C \cap \overline{P_1P_2}$ consists of $d$ points, according to Corollary \ref{transversal}. 
Since the orbit $G \cdot Q$ contains the set $C \cap (\overline{P_2P_3} \cup \overline{P_3P_1})$, it follows that the orbit $G \cdot Q$ contains at least $d^2+2d$ points. 
The orbit-stabiliser theorem implies that $|G| \ge (d+1)(d^2+2d)$. 
Then    
$$ T \ge -2+\left(1-\frac{1}{2}\right)+\left(1-\frac{1}{3}\right)+\left(1-\frac{1}{d+1}\right)=\frac{1}{6}-\frac{1}{d+1}, $$
and 
$$ d(d+1)(d+2) \le |G| = \frac{2g-2}{T} \le \frac{6d(d+1)(d-3)}{d-5}. $$
It follows that 
$$ d+2 \le 6 \times \frac{d-3}{d-5}, $$
and that $d \le 8$. 
If $d=8$, then $g=21$, namely, $C$ is smooth. 
This is a contradiction. 
We consider the case where $d=6$. 
For a triple $i, j, k$ such that $|\langle G_{P_i}, G_{P_j} \rangle (Q_k)| \ge d+1=7$, we assume that $|\langle G_{P_i}, G_{P_j} \rangle (Q_k)|=7$. 
Then the cyclic group $\langle G_{P_i}, G_{P_j} \rangle (Q_k)$ of order seven acts on six points of $C \cap \overline{P_iP_j}$. 
This implies that $\langle G_{P_i}, G_{P_j} \rangle (Q_k)$ fixes all points of $C \cap \overline{P_iP_j}$. 
With Hurwitz's theorem applied to the covering $X \rightarrow X/(\langle G_{P_i}, G_{P_j}\rangle (Q_k))$, this is a contradiction. 
Therefore, $|\langle G_{P_i}, G_{P_j} \rangle (Q_k)|\ge 8$, namely, $|G(Q)| \ge 8$. 
It follows that the orbit $G \cdot Q$ contains at least $d^2+2d=48$ points. 
The orbit-stabiliser theorem implies that $|G| \ge 8 \times 48$. 
Then    
$$ T \ge -2+\left(1-\frac{1}{2}\right)+\left(1-\frac{1}{3}\right)+\left(1-\frac{1}{8}\right)=\frac{1}{24}, $$
and 
$$ 8 \times 48 \le |G| = \frac{2g-2}{T} \le (g-1) \times 48. $$
It follows that $g=9$ or $g=10$. 
If $g=10$, then $C$ is smooth. 
Therefore, $g=9$. 
Then $|G|=8 \times 48$, the length of the orbit $G \cdot Q$ is equal to $d^2+2d=48$, and $|\langle G_{P_i}, G_{P_j} \rangle(Q_k)|=8$. 
Note that the group $\langle G_{P_i}, G_{P_j} \rangle$ is of order $48$, and acts on the set $G \cdot Q \setminus (C \cap \overline{P_iP_j})$, which consists of $48-6=42$ points. 
Then the group $\langle G_{P_i}, G_{P_j} \rangle$ has three short orbits other than $C \cap \overline{P_iP_j}$. 
Note that for such short orbits, there exist at most one such that the order of stabiliser subgroup is two. 
It follows from Fact \ref{Hurwitz} that 
$$ T \ge -2+\left(1-\frac{1}{2}\right)+\left(1-\frac{1}{3}\right) \times 2+\left(1-\frac{1}{8}\right)=\frac{17}{24}, $$
and 
$$ 48= |\langle G_{P_i}, G_{P_j} \rangle| = \frac{2g-2}{T} \le 16 \times \frac{24}{17}. $$
This is a contradiction. 
We have $d < 6$.

{\it Case 2: $m_k \ge 2$ for any $k$.} 
According to Proposition \ref{extendable}, for $k=3$, there exists a generator $\sigma \in \langle G_{P_1}, G_{P_2} \rangle \cap G_3$ considered as a linear transformation 
$$ \tilde{\sigma}(X:Y:Z)=(X:Y:\zeta Z), $$
where $\zeta$ is a primitive $m_3$-th root of unity, and   
$C$ is defined by a polynomial of the form 
$$ Z^d+G_{m_3}(X, Y)Z^{d-m_3}+\cdots+G_{d-m_3}(X, Y)Z^{m_3}+G_d(X, Y), $$
where $G_i(X, Y) \in k[X, Y]$ is a homogeneous polynomial of degree $i$. 
Since $G_d(X, Y)$ has no multiple component by Corollary \ref{transversal},  it follows that there exists a point $Q \in C \cap \overline{P_1P_2}$ with $|G_{P_3}(Q)|=m_3$. 
If $m_3=d$, then $C$ is smooth. 
Therefore, $m_3 < d$. 
Since the set of all fixed points of $\tilde{\sigma}$ coincides with $\overline{P_1P_2} \cup \{P_3\}$, it follows that $\sigma$ acts on $\frac{d}{m_3}-1$ points, namely, 
$$ m_3 \ | \ \left(\frac{d}{m_3}-1\right). $$ 
On the other hand, $m_3 \le (d/m_3)-1$. 
Then $m_3^2+m_3 \le d$. 
In particular, 
$$ m_3 < \sqrt{d} \ \mbox{ and } \ \frac{d}{m_3}>\sqrt{d}. $$
The same discussion can be applied to $m_1, m_2$.

It follows from Proposition \ref{extendable} that 
$$ G_{P_k}(Q) \supset \langle G_{P_i}, G_{P_j} \rangle \cap G_{P_k} $$
for any point $Q \in C \cap \overline{P_iP_j}$.  
Let 
$$ e_k:=\max\{|G_{P_k}(Q)|; \ Q \in C \cap \overline{P_iP_j}\}, $$
and let 
$$ l_k=e_k/m_k. $$
Then, for a point $Q \in C \cap \overline{P_iP_j}$ with $|G_{P_k}(Q)|=e_k$,  
$$ |G(Q)| \ge |\langle G_{P_i}, G_{P_j} \rangle (Q)| \times l_k \ge d l_k, $$
and the length of an orbit $G \cdot Q$ is at least 
$$ d \times \frac{d}{m_k l_k}.  $$
The orbit-stabiliser theorem implies that 
$$ |G| \ge d l_k \times d \times \frac{d}{m_k l_k}=d^2 \times \frac{d}{m_k} \ge d^2 \sqrt{d}.  $$

Assume that the number of short orbits under $G$ is at least five. 
It follows from Fact \ref{Hurwitz} (c) that 
$$ d^2 \sqrt{d} \le |G| \le 4(g-1) \le 2d(d-3). $$
Then 
$$ \sqrt{d} \le 2 \times \frac{d-3}{d} < 2 \times 1=2. $$
This implies $d < 4$. 
Assume that the number of short orbits is four. 
It follows from Fact \ref{Hurwitz} that  
$$ T \ge -2+\left(1-\frac{1}{2}\right) \times 3+\left(1-\frac{1}{d}\right)=\frac{1}{2}-\frac{1}{d}, $$
and 
$$ d^2 \sqrt{d} \le |G| = \frac{2g-2}{T} \le \frac{2d^2(d-3)}{d-2}. $$
Then 
$$ \sqrt{d} \le 2 \times \frac{d-3}{d-2} < 2. $$
This implies $d < 4$. 

Assume that the number of short orbits is three. 
We take a smooth point $Q_k \in C \cap \overline{P_iP_j}$ for any $\{i, j, k\}=\{1, 2, 3\}$. 
We consider the case where three orbits $G \cdot Q_1$, $G \cdot Q_2$, $G \cdot Q_3$ are different. 
By Proposition \ref{orders}, two of $|G(Q_1)|$, $|G(Q_2)|$, $|G(Q_3)|$ are at least $d+1$. 
The orbit-stabiliser theorem implies that $|G| \ge (d+1)d\sqrt{d}$. 
It follows from Fact \ref{Hurwitz} that 
$$ T \ge -2+\left(1-\frac{1}{d}\right)+\left(1-\frac{1}{d+1}\right) \times 2=\frac{d^2-2d-1}{d(d+1)}, $$
and 
$$ (d+1)d\sqrt{d} \le |G| = \frac{2g-2}{T} \le \frac{d^2(d+1)(d-3)}{d^2-2d-1}. $$
Then 
$$ \sqrt{d} \le \frac{d^2-3d}{d^2-2d-1} <1.$$
This implies $d <1$. 
We consider the case where there exist $i, j, k$ with $\{i, j, k\}=\{1, 2, 3\}$ such that $G \cdot Q_i=G \cdot Q_j$ and $G \cdot Q_i \ne G \cdot Q_k$. 
By Proposition \ref{orders}, $|G(Q_i)|=|G(Q_j)| \ge d+1$ and $|G(Q_k)| \ge d$.  
The orbit-stabiliser theorem implies that $|G| \ge (d+1)d\sqrt{d}$.
Then 
$$ T \ge -2+\left(1-\frac{1}{2}\right)+\left(1-\frac{1}{d}\right)+\left(1-\frac{1}{d+1}\right)=\frac{d^2-3d-2}{2d(d+1)}, $$
and 
$$ (d+1)d\sqrt{d} \le |G| = \frac{2g-2}{T} \le \frac{2d^2(d+1)(d-3)}{d^2-3d-2}. $$
It follows that 
$$ \sqrt{d} \le \frac{2d^2-6d}{d^2-3d-2} \le \frac{9}{4}.$$
This implies that $d < 6$.

Finally, we consider the case where $G \cdot Q_1=G \cdot Q_2=G \cdot Q_3$. 
By Proposition \ref{orders}, $|G(Q_k)| \ge d+1$ for $k=1, 2, 3$. 
Let 
$$ e:= \max\{e_1, e_2, e_3\}. $$
We take $k \in \{1, 2, 3\}$ with $e_k=e$ and a point $Q \in C \cap \overline{P_iP_j}$ with $|G_{P_k}(Q)|=e_k$.  
Then $|G(Q)| \ge d l_k \ge 2d$ if $l_k \ge 2$, where $l_k=e_k/m_k$. 
It follows from Fact \ref{unramified} that if the tangent line of a point $R \in X$ contains $P_i$, then the tangent line does not contain $P_j$ for $j \ne i$.  
This implies that 
$$ \left(\bigcup_{R \in C \cap \overline{P_jP_k}} G_{P_i} \cdot R\right) \bigcap \left(\bigcup_{R \in C \cap \overline{P_iP_k}}G_{P_j} \cdot R\right) =\emptyset $$ 
if $i \ne j$. 
Assume that $l_k \ge 2$. 
The length of an orbit $G \cdot Q$ is at least 
$$  3 \times d \times \frac{d}{m_k l_k}. $$ 
The orbit-stabiliser theorem implies that
$$ |G| \ge d l_k \times 3d \times \frac{d}{m_k l_k} = 3d^2 \times \frac{d}{m_k} \ge 3d^2 \sqrt{d}. $$
It follows that 
$$ T \ge -2+\left(1-\frac{1}{2}\right)+\left(1-\frac{1}{3}\right)+\left(1-\frac{1}{d l_k}\right) \ge \frac{1}{6}-\frac{1}{2d}, $$
and 
$$ 3d^2\sqrt{d} \le |G| = \frac{2g-2}{T} \le d(d-3) \times \frac{6d}{d-3}=6d^2. $$
Then $\sqrt{d} \le 2$. 
This implies $d \le 4$. 
We consider the case where $l_k=1$.  
The length of an orbit $G \cdot Q$ is at least 
$$  3 \times d \times \frac{d}{m_k}. $$ 
The orbit-stabiliser theorem implies that
$$ |G| \ge (d+1) \times 3d \times \frac{d}{m_k} \ge 3d(d+1) \times \sqrt{d}. $$
It follows that 
$$ T \ge -2+\left(1-\frac{1}{2}\right)+\left(1-\frac{1}{3}\right)+\left(1-\frac{1}{d+1}\right) = \frac{1}{6}-\frac{1}{d+1}, $$
and 
$$ 3d(d+1)\sqrt{d} \le |G| = \frac{2g-2}{T} \le \frac{6d(d+1)(d-3)}{d-5}. $$
Then 
$$ \sqrt{d} \le 2 \times \frac{d-3}{d-5}, $$
and $d \le 9$. 
We recall that $m_k < d$, and the integer $m_k$ divides the integers $d$ and $(d/m_k)-1$. 
For the case $d=9$, $m_k=3$ divides $\frac{9}{3}-1=2$. 
This is a contradiction. 
For the case $d=8$, $m_k=2$ divides $\frac{8}{2}-1=3$. 
This is a contradiction. 
We consider the case where $d=6$. 
Then we have $m_k=2$. 
This implies that $|G(Q_k)|$ is even, namely, $|G(Q_k)| \ge 8$. 
The length of an orbit $G \cdot Q$ is at least 
$$  3 \times d \times \frac{d}{m_k}=54. $$ 
The orbit-stabiliser theorem implies that
$$ |G| \ge 8 \times 54.  $$
It follows that 
$$ T \ge -2+\left(1-\frac{1}{2}\right)+\left(1-\frac{1}{3}\right)+\left(1-\frac{1}{8}\right) =\frac{1}{24}, $$
and 
$$ 8 \times 54 \le |G| = \frac{2g-2}{T} \le (g-1) \times 48. $$
Then $g=10$, namely, $C$ is smooth. 
This is a contradiction. 
We have $d < 6$. 

\section{Proof of Theorem: The case $d=4$}
It follows from Fact \ref{smooth} that $g=1$ or $g=2$. 
Assume that $g=2$. 
Then there exists a unique singular point $Q \in C$ with multiplicity $2$. 
It follows from Corollary \ref{transversal} that any line containing $Q$ does not contain two outer Galois points. 
Since three outer Galois points $P_1, P_2, P_3$ exist, there exists $i$ such that $C \cap \overline{P_iQ} \setminus \{Q\}$ consists of exactly two points. 
Let $C \cap \overline{P_iQ}\setminus\{Q\}=\{R_1, R_2\}$. 
The two points over $Q$ for the normalisation $X \rightarrow C$ are denoted by $Q_1, Q_2$.  
Since the smooth model $X$ is hyperelliptic, the projection from $Q$ corresponds to the canonical linear system, namely, $R_1+R_2$ is a canonical divisor. 
Note that $R_i+Q_j$ is not a canonical divisor for $i=1, 2$ and $j=1, 2$. 
Therefore,  for any element $\sigma \in G_{P_i}$,  $\sigma(\{R_1, R_2\})=\{R_1, R_2\}$ or $\{Q_1, Q_2\}$. 
Since $|G_{P_i}|=4$, it follows that there exists $\sigma \in G_{P_i}$ with $\sigma(\{R_1, R_2\})=\{Q_1, Q_2\}$, namely, $Q_1+Q_2$ is a canonical divisor. 
This implies that a tangent line $T_Q$ of $C$ at $Q$ is uniquely determined, which corresponds to the effective divisor $D:=2Q_1+2Q_2$. 
According to Riemann--Roch's theorem, $\dim |D|=2$, namely, the linear system corresponding to a birational embedding into $\mathbb{P}^2$ is complete. 
Since any $\sigma \in G_{P_i}$ fixes the divisor induced by any line passing through $P_i$, any $\sigma \in G_{P_i}$ is the restriction of some linear transformation $\tilde{\sigma}$ of $\mathbb{P}^2$. 
Then $\tilde{\sigma}(T_Q)=T_Q$, namely, $G_{P_i}$ acts on the set $\{Q_1, Q_2\}$. 
As we saw above, there exists $\sigma \in G_{P_i}$ with $\sigma(\{R_1, R_2\})=\{Q_1, Q_2\}$.  
This is a contradiction. 

Assume that $g=1$. 
As we saw in the previous section, there exist $i, j$ such that $|\langle G_{P_i}, G_{P_j}\rangle(Q)| \ge d+1$ for a point $Q \in C \cap \overline{P_iP_j}$. 
A well known theorem on the automorphism group of an elliptic curve implies that $|\langle G_{P_i}, G_{P_j}\rangle(Q)| =6$ (see, for example, \cite[III. Theorem 10.1]{silverman}). 
If $G_{P_i}$ is a cyclic group of order $4$, then there exists a point $Q' \in X$ such that $|G_{P_i}(Q')|=4$. 
This is a contradiction, because there does not exist an elliptic curve admitting two cyclic coverings of degree $6$ and of degree $4$ with totally ramified points (see, for example, \cite[III. Theorem 10.1]{silverman}). 
Therefore, $G_{P_i} \cong G_{P_j} \cong (\mathbb{Z}/2\mathbb{Z})^{\oplus 2}$. 
Note that there exist two involutions $\sigma \in G_{P_i}$ with $X/\langle \sigma \rangle \cong \mathbb{P}^1$, since there exist eight ramification points for the covering $X \rightarrow X/G_{P_i}$.  
Since $G_{P_i}$, $G_{P_j}$ act on four points of $C \cap \overline{P_iP_j}$ transitively, there exist involutions $\sigma \in G_{P_i}$, $\tau \in G_{P_j}$ such that $X/\langle \sigma \rangle \cong \mathbb{P}^1$, $X/\langle \tau \rangle \cong \mathbb{P}^1$, and $\sigma|_{C \cap \overline{P_iP_j}}=\tau|_{C \cap \overline{P_iP_j}}$. 
We take four points $Q_1, Q_2, Q_3, Q_4 \in C \cap \overline{P_iP_j}$ so that $\sigma(Q_1)=Q_2$, $\sigma(Q_3)=Q_4$, $\tau(Q_1)=Q_2$ and $\tau(Q_3)=Q_4$. 
Then the double coverings $X \rightarrow X/\langle \sigma \rangle \cong \mathbb{P}^1$ and $X \rightarrow X/\langle \tau \rangle \cong \mathbb{P}^1$ are given by rational functions $f \in k(X)$ and $g \in k(X)$ such that 
$$ (f)=Q_1+Q_2-Q_3-Q_4 \ \mbox{ and } \ (g)=Q_1+Q_2-Q_3-Q_4 $$
respectively.  
Then $k(f)=k(g)$, namely, $\sigma=\tau$, and $G_{P_i} \cap G_{P_j} \supset \{\sigma\}$. 
This is a contradiction.     

\begin{center}
{\bf Acknowledgements}
\end{center} 
The author is grateful to Professor Takeshi Harui for helpful comments enabling the author to prove Theorem in the case $d=6$.


\begin{thebibliography}{100} 
\bibitem{duyaguit-miura} C. Duyaguit and K. Miura, On the number of Galois points for plane curves of prime degree, Nihonkai Math. J. {\bf 14} (2003), 55--59. 

\bibitem{fukasawa1} S. Fukasawa, Galois points for a plane curve in arbitrary characteristic, Proceedings of the IV Iberoamerican Conference on Complex Geometry, Geom. Dedicata {\bf 139} (2009), 211--218. 

\bibitem{fukasawa2} S. Fukasawa, Classification of plane curves with infinitely many Galois points, J. Math. Soc. Japan {\bf 63} (2011), 195--209. 

\bibitem{fukasawa3} S. Fukasawa, On the number of Galois points for a plane curve in characteristic zero, preprint, arXiv:1604.01907. 

\bibitem{fukasawa4} S. Fukasawa, Algebraic curves admitting non-collinear Galois points, Rend. Sem. Mat. Univ. Padova, to appear. 

\bibitem{fukasawa5} S. Fukasawa, Galois points and rational functions with small value sets, Hiroshima Math. J., to appear. 

\bibitem{hkt} J. W. P. Hirschfeld, G. Korchm\'{a}ros and F. Torres, Algebraic Curves over a Finite Field, Princeton Univ. Press, Princeton, 2008.  

\bibitem{miura-yoshihara} K. Miura and H. Yoshihara, Field theory for function fields of plane quartic curves, J. Algebra {\bf 226} (2000), 283--294.

\bibitem{silverman} J. H. Silverman, The Arithmetic of Elliptic Curves, Graduate Texts in Mathematics {\bf 106}, Springer-Verlag, New York, 1986.  

\bibitem{stichtenoth} H. Stichtenoth, Algebraic Function Fields and Codes, Universitext, Springer-Verlag, Berlin, 1993.

\bibitem{yoshihara} H. Yoshihara, Function field theory of plane curves by dual curves, J. Algebra {\bf 239} (2001), 340--355.

\end{thebibliography}
\end{document}